\documentclass[]{amsart}

\usepackage{amsmath,amssymb,amsthm}
\usepackage{graphicx,tikz}

\newtheorem*{thm}{Theorem}
\newtheorem*{proposition}{Proposition}
\newtheorem{corollary}{Corollary}

\newtheorem{lemma}{Lemma}

\theoremstyle{definition}

\theoremstyle{remark}

\DeclareMathOperator{\Imn}{Im}
\DeclareMathOperator{\Rem}{Re}

\begin{document}

\title[]{On the Stability of Fourier Phase Retrieval}
\keywords{Phase Retrieval, Stability, Fourier Transform}
\subjclass[2010]{42A63, 45Q05} 
\thanks{S.S. is supported by the NSF (DMS-2123224) and the Alfred P. Sloan Foundation.}

\author[]{Stefan Steinerberger}
\address{Department of Mathematics, University of Washington, Seattle, WA 98195, USA}
\email{steinerb@uw.edu}

\begin{abstract} Phase retrieval is concerned with recovering a function $f$ from the absolute value of its Fourier transform $|\widehat{f}|$. We study the stability properties of this problem in Lebesgue spaces. Our main results shows that
$$  \| f-g\|_{L^2(\mathbb{R}^n)}  \leq   2\cdot \| |\widehat{f}| - |\widehat{g}| \|_{L^2(\mathbb{R}^n)} +   h_f\left( \|f-g\|^{}_{L^p(\mathbb{R}^n)}\right) + J(\widehat{f}, \widehat{g}),$$
where $1 \leq p < 2$, $h_f$ is an explicit nonlinear function depending on the smoothness of $f$ and $J$ is an explicit term capturing the invariance under translations. A noteworthy aspect is that the stability is phrased in terms of $L^p$ for $1 \leq p < 2$ while, usually, $L^p$ cannot be used to control $L^2$, the stability estimate has the flavor of an inverse H\"older inequality. It seems conceivable that the estimate is optimal up to constants.
\end{abstract}

\maketitle

\section{Introduction and Results}
\subsection{Introduction} Phase retrieval refers to a broad class of problems where one is given incomplete information about an object (often the size of the coefficients with respect to some basis expansion but not their phase) and then tries to reconstruct the object. In the case of the Fourier transform, the challenge is to recover a function $f$ from knowing only the modulus of its Fourier transform $|\widehat{f}|$. The problem itself is classical and first arose, implicitly, a century ago in the setting of x-ray crystallography. It has since appeared in a variety of different fields \cite{dainy, ha, jam, miao, mill, sh, walther}. 
There is a vast literature, one possible starting point is \cite{akut, akut2, bu, bu2, bruck, cr, cr2, fann, fienup, gabor, green, grohs, hof, jam1, jam2, kil, rose, sanz}.  The question has recently been studied in more abstract settings (say, recovering Hilbert space elements for which one knows the size of certain inner products), see \cite{al, al2, alex, bal, bal2, bal3, ban, bar, cahill, eldar, grohs1, grohs2, mallat} and references therein. We cannot possibly hope to summarize the existing literature but we emphasize two very recent and excellent surveys, one by Grohs, Koppensteiner \& Rathmair \cite{grohs} about theoretical aspects and one by Fannjiang \& Strohmer \cite{fann} about the numerical side of things.\\

We study the stability problem and begin by recalling the translation symmetry: the functions $f(x)$ and its shift $f(x+\varepsilon)$ cannot be distinguished from looking at the modulus of their Fourier transform and this has to play a role in all the stability results. We also recall that  if $f, g \in L^2(\mathbb{R})$ are both compactly supported, then there exists a convenient characterization of all pairs $(f,g) \in L^2(\mathbb{R}) \times L^2(\mathbb{R})$ for which $|\widehat{f}| = |\widehat{g}|$ in terms of complex analysis (since both $\widehat{f}$ and $\widehat{g}$ are entire) \cite{akut, akut2, hof, walther}. As a consequence of this characterization, we have a basic uniqueness result, see for example \cite[Theorem 4.9]{grohs} or \cite[Proposition 3.3]{kil} 

\begin{proposition}[see e.g. \cite{akut, akut2, grohs, hof, kil, walther}]   If $f \in L^2(\mathbb{R})$ is compactly supported and satisfies 
$$\widehat{f}(-\xi) = \overline{\widehat{f}(\xi)},$$ then it is uniquely determined by $|\widehat{f}|$. \end{proposition}

We are interested in the stability question: let us fix $f,g \in L^2(\mathbb{R})$. If $|\widehat{f}| \sim |\widehat{g}|$ are close in $L^2$, does this necessarily imply that $f$ and $g$ themselves are close in $L^2$? Without any further assumptions, this is certainly wrong: for any function $h:\mathbb{R} \rightarrow \left\{-1,1\right\}$, we can define
$$ \widehat{g}(\xi) = h(\xi) \widehat{f}(\xi)$$
which results in them having the same modulus but there is absolutely no reason for them to be close to one another in $L^2$. This shows that at least one more assumption is needed. We will show that such stability results become possible if we assume that $f$ and $g$ are close in $L^p$ for some $1 \leq p < 2$.

\section{The Result}
\subsection{The simplest case.}
We present the principle first in its simplest form. This is not the most general formulation but maybe the one that is most easily visualized.
\begin{quote}
\textbf{Fact.} For any two functions $f,g \in L^1(\mathbb{R}^n) \cap L^2(\mathbb{R}^n)$, if
\begin{enumerate}
\item  $\mbox{supp}(\widehat{f})$ has finite measure,
\item $g$ is close to $f$ in $L^1$ and 
\item  $|\widehat{f}|$ is close to $|\widehat{g}|$ in $L^2$,
\end{enumerate}
 then $f$ is close to $g$ in $L^2$ (up to the translation symmetry).
\end{quote}

\begin{center}
\begin{figure}[h!]
\begin{tikzpicture}[scale=1]
\draw [thick, <->] (-5, 0) -- (5,0);
\draw [ultra thick] (-3, 0) to[out=30, in =180] (-0.2,0.5);
\draw [thick, ->] (0,0) -- (0,3.5);
\draw [ultra thick] (0.2, 0.5) to[out=0, in =150] (3,0);
\draw [ultra thick] (-0.2, 0.5) -- (-0.1, 3) -- (0.1, 3) -- (0.2, 0.5);
\draw [ultra thick] (-2.7, 0) to[out=30, in =180] (0.1,0.5);
\draw [ultra thick] (0.5, 0.5) to[out=0, in =150] (3.3,0);
\draw [ultra thick] (0.1, 0.5) -- (0.2, 3) -- (0.4, 3) -- (0.5, 0.5);
\end{tikzpicture}
\caption{An even function $f$ and $g(x) = f(x-\varepsilon)$. They are close in $L^1$ and $|\widehat{f}| = |\widehat{g}|$ but they are not close in $L^2$. Any stability estimate needs to compensate for this translation symmetry.}
\end{figure}
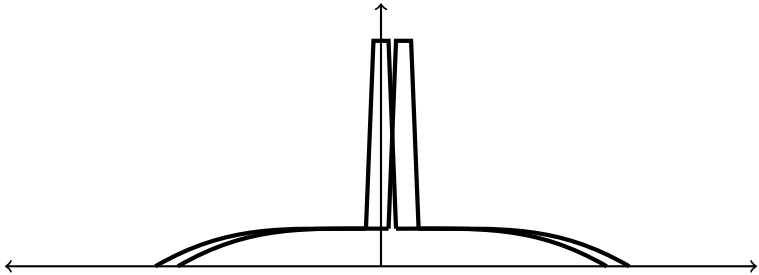
\end{center}
\vspace{-15pt}

We quickly describe how this could be interpreted. Let us first specify what it means for two function $f,g$ to be close in $L^1$ but not in $L^2$: it indicates that $f-g$ has concentrations of $L^1-$mass or, equivalently, that $|f-g|$ assumes large values over a small interval scaled in such a way that the contribution in $L^1$ is not substantial but becomes substantial in $L^2$. One way of phrasing the general principle is that in such a case we are either dealing with a translation symmetry, $g$ is close to a shift of $f$, or this behavior becomes visible in $\| |\widehat{f}| - |\widehat{g}| \|_{L^2}$: strong localized mass translates into slow decay of the Fourier transform.  Assuming a smoothness condition on $f$ we obtain for large frequencies $\xi$ that $||\widehat{f}(\xi)| - |\widehat{g}(\xi)|| \sim | \widehat{g}(\xi)|$ and we can recover the slow decay this way. Naturally, there are other interpretations.
We introduced this principle in the case where $f$ satisfies a very strong smoothness condition and has a real-valued Fourier transform that is compactly supported. The general result does not require any such conditions.

\begin{corollary} Let $f \in L^1(\mathbb{R}^n) \cap L^2(\mathbb{R}^n)$ have a real-valued Fourier transform supported on a set of measure 
$ L = | \{\xi \in \mathbb{R}^n: \widehat{f}(\xi) \neq 0\}|.$
For all $g \in L^1(\mathbb{R}^n) \cap L^2(\mathbb{R}^n)$ 
 $$ \| f-g\|_{L^2(\mathbb{R}^n)} \leq  2\cdot  \| |\widehat{f}| - |\widehat{g}| \|_{L^2(\mathbb{R}^n)}  + 30 \sqrt{L} \cdot \|f-g\|_{L^1(\mathbb{R}^n)}+  2\| \Imn \widehat{g} \|_{L^2(\mathbb{R}^n)}.$$
\end{corollary}
This result shows stability of the phase retrieval problem of a function $f$ with real-valued and compactly supported Fourier transform in $L^1(\mathbb{R}^n) \cap L^2(\mathbb{R}^n)$ up to the translation symmetry. If the Fourier transform was not real-valued and merely compactly supported, we obtain the same result with $\| \Imn \widehat{g} \|_{L^2}$, the term accounting for the translation symmetry, replaced by a slightly more general expression which we discuss below.
 It is not difficult to see that Corollary 1 has the sharp scaling and the optimal dependence on $L$: let $\phi:\mathbb{R} \rightarrow \mathbb{R}$ be an even, nonnegative and compactly supported $C^{\infty}-$function in $[-1,1]$ and set
 $$ \widehat{f}(\xi) = \frac{1}{L}\phi\left(\frac{\xi}{L}\right)  \quad \mbox{and} \quad \widehat{g}(\xi) = -f(\xi).$$
We have $|\widehat{f}| = |\widehat{g}|$, $\Imn \widehat{g} \equiv 0$,
$ f(x) = \widehat{\phi}(Lx)$ and thus, as $L$ becomes large,
$$ \|f-g\|_{L^2} \sim \frac{1}{\sqrt{L}} \qquad \mbox{and} \qquad \|f-g\|_{L^1} \sim \frac{1}{L}.$$
Therefore Corollary 1 is optimal up to constants.
Once $f$ is less smooth, the $L^1-$distance stops acting linearly, the bound moves from Lipschitz to H\"older. 
\begin{corollary} Let $f \in L^1(\mathbb{R}^n)\cap L^2(\mathbb{R}^n)$ have a real-valued Fourier transform and have its $k-$th derivative in $L^1(\mathbb{R})$ where $k>(n+2)/2$. Then, for some constant $c_f>0$ depending only on $f$ and all $g \in L^1(\mathbb{R}^n) \cap L^2(\mathbb{R}^n)$ ,
 $$ \| f-g\|_{L^2(\mathbb{R}^n)} \leq    2\cdot \| |\widehat{f}| - |\widehat{g}| \|_{L^2(\mathbb{R}^n)}  + c_f  \|f-g\|^{1-\frac{n}{2k}}_{L^1(\mathbb{R}^n)}+  2\| \Imn \widehat{g} \|_{L^2(\mathbb{R}^n)}.$$
\end{corollary}
\begin{center}
\begin{figure}[h!]
\begin{tikzpicture}[scale=1]
\node at (0,0) {\includegraphics[width=0.6\textwidth]{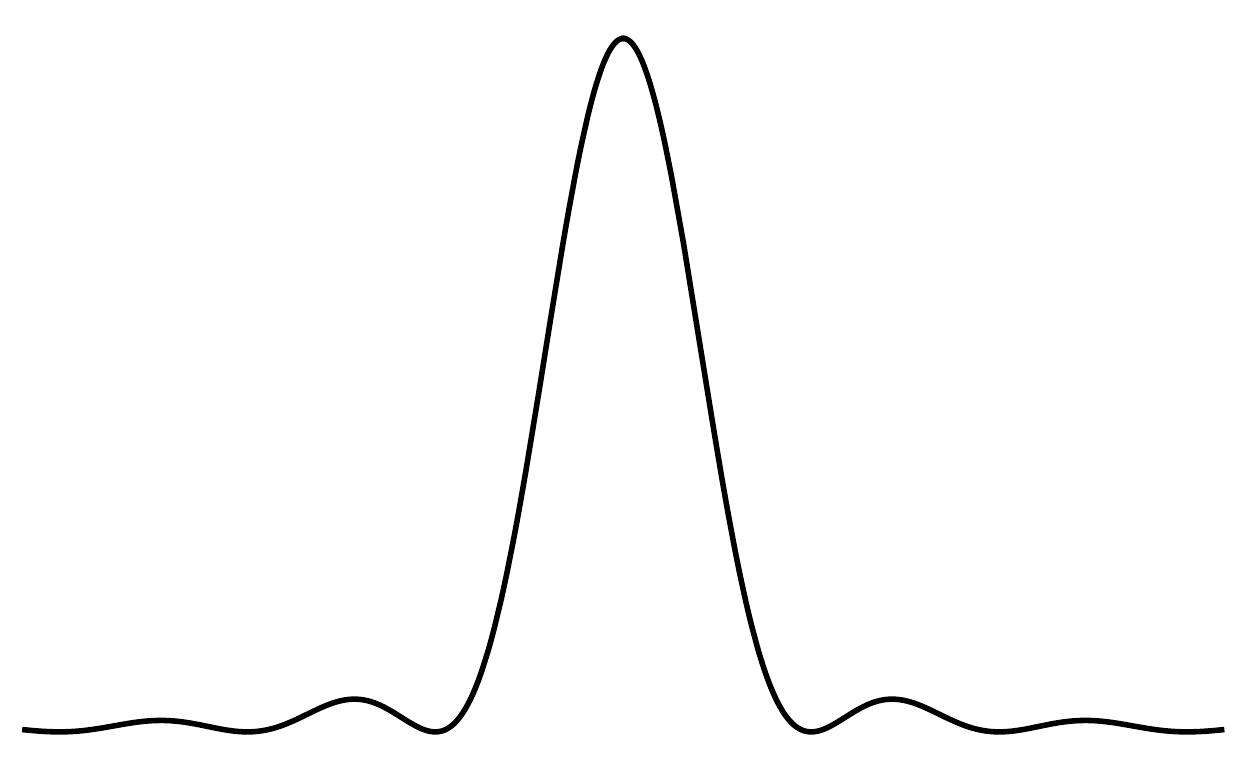}};
\draw [thick,<->] (-4,-2.1) -- (4,-2.1);
\end{tikzpicture}
\caption{$(1-\cos{x})x^{-2}$ has super-linear stability in $L^1$.}
\end{figure}
\end{center}
The result is also applicable to specific functions (see Fig. 2). We discuss
$ f(x) =(1-\cos{x})x^{-2}$ satisfying $\widehat{f}(\xi) = \max\left\{0,1-|x|\right\}$. An application of the Theorem shows that for some universal constant $c>0$ and any even function $g:\mathbb{R} \rightarrow \mathbb{R}$,
$$  \| f-g\|_{L^2(\mathbb{R})}  \leq    2 \cdot \| |\widehat{f}| - |\widehat{g}| \|_{L^2(\mathbb{R})} +  c \min\left\{ \|f-g\|^{3/2}_{L^1(\mathbb{R})},  \|f-g\|^{}_{L^1(\mathbb{R})}\right\}.$$
The second estimate follows from Corollary 1 since $f$ is band-limited. The first bound, decaying faster than linearly for $f$ close to $g$ in $L^1$, follows from using the explicit form of the Theorem. This inequality could also easily be directly proven since all the terms are real. We see that the Fourier Phase Retrieval problem for $f$ is quite stable in the space of symmetric perturbations as soon as $f$ and $g$ are quite close in $L^1$ (and $g$ is even; otherwise the translation symmetry makes any approximation faster than linear impossible).

\subsection{The General Result.} 
We will now describe the general stability result for functions in $L^p(\mathbb{R}^n) \cap L^2(\mathbb{R}^n)$ where $1 \leq p < 2$. 
The result has the same form as the results as above: the $L^2-$distance of two functions $f$ and $g$ is bounded from above by the sum of three terms: (a) the $L^2-$distance of $|\widehat{f}|$ and $|\widehat{g}|$, (b) the distance of $f$ and $g$ in $L^p$ where $1 \leq p <2$, in a  way depending on the smoothness of $f$, and (c) a term accounting for the invariance under translations.

\begin{thm} Let $1 \leq p < 2$ and  $f \in L^p(\mathbb{R}^n) \cap L^2(\mathbb{R}^n)$. We define $h_f:\mathbb{R}_{\geq 0 } \rightarrow \mathbb{R}_{\geq 0 }$
$$ h_f(x) =  \left(8 \int_{|\widehat{f}(\xi)| \leq 10x}  | \widehat{f}(\xi)  |^2 d\xi\right)^{1/2} + \begin{cases} x \qquad &\mbox{if}~~p>1 \\ 0 \qquad &\mbox{if}~ p=1.\end{cases}$$
Then, for all $g \in L^1(\mathbb{R}^n) \cap L^2(\mathbb{R}^n)$,
$$  \| f-g\|_{L^2(\mathbb{R}^n)}  \leq   2\cdot \| |\widehat{f}| - |\widehat{g}| \|_{L^2(\mathbb{R}^n)} +   h_f\left( \|f-g\|^{}_{L^p(\mathbb{R}^n)}\right) + 2 \left\|  \Imn  \overline{\widehat{f}} |\widehat{f}|^{-1}\widehat{g} \right\|_{L^2(\mathbb{R}^n)}^2.$$
\end{thm}
The result is very much in the same flavor as the results above: we have quantitative dependence on $\|f-g\|_{L^p}$ whose rate
depends on the smoothness of $f$. If $\widehat{f}$ is supported on a set of measure $L$, then $h_f$ grows at most linearly since
$$  \left(8 \int_{|\widehat{f}(\xi)| \leq 10x}  | \widehat{f}(\xi)  |^2 d\xi\right)^{1/2} \leq \left(8 \int_{|\widehat{f}(\xi)| \leq 10x} (10x)^2d\xi \right)^{1/2} \leq 30 \sqrt{L} x.$$
In the special case $p=1$, it might even have smaller growth (as seen in the example above).
If $\widehat{f}$ is real-valued, then the last term simplifies to
$$  \Imn  \overline{\widehat{f}} |\widehat{f}|^{-1}\widehat{g} =  \Imn \widehat{g}$$
which recovers the previous results. We quickly illustrate the meaning of this term for smooth functions $f,g \in L^2$ by considering the case $g(x) = f(x-\varepsilon)$. Their Fourier transforms have the same modulus. As $\varepsilon \rightarrow 0$ and for fixed $\xi \in \mathbb{R}$ the term $1-\cos{(\varepsilon \xi)}$ is quadratic in $\varepsilon$ while $i\sin{(\varepsilon \xi)}$ is linear in $\varepsilon$ and thus
\begin{align*}
\|f - g\|_{L^2} &= \| \widehat{f}(\xi) - e^{- i\varepsilon \xi} \widehat{f}(\xi)\|_{L^2} =
\| \widehat{f}(\xi) \left(1 - \cos{(\varepsilon \xi)} + i \sin{(\varepsilon \xi)}\right) \||_{L^2}\\
&\sim \| \widehat{f}(\xi) \left(i \sin{(\varepsilon \xi)}\right) \||_{L^2} =  \| \widehat{f}(\xi) \sin{(\varepsilon \xi)}\|_{L^2}
\end{align*}
We see that this term is a genuine $L^2-$quantity that is unlikely to be controlled by $L^p$ for $p<2$. Indeed, it is controlled by the third quantity, since 
$$ | \Imn \overline{\widehat{f}(\xi)} |\widehat{f}(\xi)|^{-1}\widehat{g}(\xi) |  =   | \Imn  \overline{\widehat{f}(\xi)} |\widehat{f}(\xi)|^{-1}\widehat{f}(\xi) e^{-i \varepsilon \xi} | = |\widehat{f}(\xi)\sin{(\varepsilon \xi)}|.$$

\subsection{Main Idea and Extensions.} We quickly illustrate the main idea and how it would allow for even more general results. We focus on the case $p=1$ and $f,g \in L^1(\mathbb{R}) \cap L^2(\mathbb{R})$. Let us suppose 
$ \varepsilon = \|f - g \|_{L^1(\mathbb{R})}$
is small. Then 
$$| \widehat{f}(\xi) - \widehat{g}(\xi)|  \leq \|\widehat{f} - \widehat{g}\|_{L^{\infty}} \leq \|f-g\|_{L^1} = \varepsilon$$
 is uniformly small in $\xi$. Let us now consider a value $\xi$ where $|\widehat{f}(\xi)| \geq 10 \varepsilon$.
 
\begin{center}
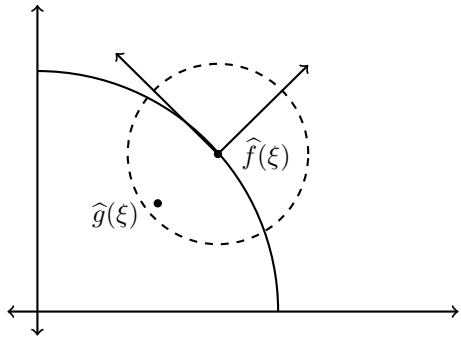
\begin{figure}[h!]
\begin{tikzpicture}[scale=0.8]
\draw [thick,<->] (-4,-2.1) -- (3.5,-2.1);
\draw [thick, <->] (-3.5,-2.5) -- (-3.5, 3);
\filldraw (-0.5,0.52) circle (0.06cm);
\node at (0.3, 0.52) {$\widehat{f}(\xi)$};
\draw [thick, dashed] (-0.5, 0.52) circle (1.5cm);
\draw [thick] (0.5,-2.1) arc (0:90:4);
\draw [thick, ->] (-0.5, 0.52) -- (1, 2);
\draw [thick, ->] (-0.5, 0.52) -- (-2.2, 2.2);
\filldraw (-1.5, -0.3) circle (0.06cm);
\node at (-2.2, -0.5) {$\widehat{g}(\xi)$};
\end{tikzpicture}
\caption{A sketch of the main idea. }
\end{figure}
\end{center}

We can express $\widehat{g}(\xi) = \widehat{f}(\xi) + (\widehat{g}(\xi) - \widehat{f}(\xi))$ and then consider the quantity $ (\widehat{g}(\xi) - \widehat{f}(\xi))$ as a vector in $\mathbb{R}^2 \cong \mathbb{C}$. It is relatively small compared to the size of $\widehat{f}(\xi)$. If it points roughly in the direction of $\widehat{f}(\xi)$, then $|\widehat{f}(\xi) - \widehat{g}(\xi)| \sim | |\widehat{f}(\xi)| - |\widehat{g}(\xi)||$ and the difference in $L^2$ shows up in the modulus. If that is not the case, then $\widehat{g}(\xi) - \widehat{f}(\xi)$ points roughly in
the direction $i\widehat{f}(\xi)$ which means that there is a nontrivial contribution to the subspace corresponding to translations of $f$. 
What is interesting about this idea is that very few properties of the absolute value $|\cdot|: \mathbb{C} \rightarrow \mathbb{R}$ are being used (though the fact that the `critical' subspace that cannot be recovered corresponds to translations of the function $f$ is very much connected to using the absolute value; the more general problem will have other `critical' subspaces that do not have such an easy interpretation). Under some regularity assumptions, the same type of arguments could be used to deal with more general problems of this type. This would allow one, for example, to define $h:\mathbb{C} \rightarrow \mathbb{R}$ via 
$$ h(z) = ((\Rem z)^4 + (\Imn z)^4)^{1/4}$$
and then try to study the phase retrieval problem for $h(\widehat{f}) \sim h(\widehat{g})$ for which similar stability estimates could be obtained. The classical Fourier phase retrieval is well motivated and this is maybe not (or not yet) the case for this generalized problem; however, it does seem interesting that the methods extend.

\section{Proofs}
\subsection{A Lemma.}
\begin{lemma} For all $0 \leq w \in \mathbb{R}$ and all $z \in \mathbb{C}$ satisfying $|z-w| \leq |w|/2$
$$ | w - \Rem z|^2 \leq | w  - |z| |^2 + 2 \left| \frac{z-w}{w} \right| \cdot | \Imn z|^2.$$
\end{lemma}
\begin{proof}
Both sides of the inequality are invariant under multiplication with scalars, so we can assume w.l.o.g. that $w=1$. It then remains to show that
$$ | 1 - \Rem z|^2 \leq |1 -  |z|  |^2 + 2 \left| z - 1\right| \cdot | \Imn z|^2 \qquad \mbox{for all}~|z-1| \leq\frac{1}{2}.$$
We make the ansatz $z = (1+x) + i y$ which reduces the desired inequality to
$$ 2 + 2x + y^2 + 2y^2 \sqrt{x^2+y^2}\geq 2\sqrt{(1+x)^2+y^2} $$
for all $ x^2 + y^2 \leq r^2$ for some $r$ to be determined. The left-hand side is certainly positive in the regime that we consider, so we can square both sides and try to verify that identity instead. The desired inequality factors into
$$ X = y^2\left( y^2 + 8\sqrt{x^2+y^2} + 4x + 4x^2 y^2 + 8x\sqrt{x^2+y^2} + 4y^2\sqrt{x^2+y^2}  \right) \geq 0.$$
We define $r$ via $x^2+y^2 = r^2$ and note that for any $r \leq 1/2$
\begin{align*}
\frac{X}{y^2}&\geq 8 \sqrt{x^2+y^2} + 4x + 8x \sqrt{x^2+y^2}   \geq 4r - 8 r^2 \geq 0.
\end{align*}
\end{proof}


\subsection{Proof of the Theorem}
\begin{proof}
Let $1 \leq p < 2$. 
We fix the parameter 
$$ \varepsilon = \|f-g\|_{L^p(\mathbb{R}^n)}$$
and write
$$ \|f-g\|_{L^2(\mathbb{R}^n)}^2 = \int_{\mathbb{R}^n}{ | \widehat{f}(\xi) - \widehat{g}(\xi) |^2 d\xi}.$$
We split this integral into two integrals over disjoint regions
\begin{align*}
 \int_{\mathbb{R}^n}{ | \widehat{f}(\xi) - \widehat{g}(\xi) |^2 d\xi} &= \int_{|\widehat{f}(\xi)| \geq 10\varepsilon}  | \widehat{f}(\xi) - \widehat{g}(\xi) |^2 d\xi\\
 &+ \int_{|\widehat{f}(\xi)| \leq 10\varepsilon}  | \widehat{f}(\xi) - \widehat{g}(\xi) |^2 d\xi.\end{align*}
The remainder of the argument is comprised of estimating these two integrals.\\

\textit{First Integral.}
We split the first integral once more
\begin{align*} \int_{|\widehat{f}(\xi)| \geq 10\varepsilon}  | \widehat{f}(\xi) - \widehat{g}(\xi) |^2 d\xi &=
\int_{|\widehat{f}(\xi)| \geq 10\varepsilon \atop \left|\widehat{f}(\xi) - \widehat{g}(\xi)\right| \leq \varepsilon}  | \widehat{f}(\xi) - \widehat{g}(\xi) |^2 d\xi\\
&+ \int_{\left|\widehat{f}(\xi)\right| \geq 10\varepsilon \atop \left|\widehat{f}(\xi) - \widehat{g}(\xi)\right| \geq \varepsilon}  | \widehat{f}(\xi) - \widehat{g}(\xi) |^2 d\xi. \end{align*}
\textit{First term.} We start by analyzing the first term. 
Fix a $\xi \in \mathbb{R}$ such that $|\widehat{f}(\xi)| \geq 10\varepsilon$ and $|\widehat{f}(\xi) - \widehat{g}(\xi)| \leq \varepsilon$. We interpret $\widehat{f}(\xi)$ (which is not 0) and $i\widehat{f}(\xi)$ as the directions of two orthogonal vectors in $\mathbb{R}^2$ and use them to express
$$ \widehat{g}(\xi) = \widehat{f}(\xi) + a \frac{\widehat{f}(\xi)}{|\widehat{f}(\xi)|} + b \frac{i\widehat{f}(\xi)}{|\widehat{f}(\xi)|} $$
for some unique $a,b \in \mathbb{R}$. We see that $a$ and $b$ cannot be very large since
$$ \sqrt{a^2 + b^2} = | \widehat{f}(\xi) - \widehat{g}(\xi)| \leq \varepsilon \leq \frac{|\widehat{f}(\xi)|}{10}.$$
Multiplying on both sides with $\overline{\widehat{f}(\xi)}/|\widehat{f}(\xi)|$ results in
$$ \frac{\overline{\widehat{f}(\xi)}}{|\widehat{f}(\xi)|} \widehat{g}(\xi) = |\widehat{f}(\xi)| + a + b i.$$
However, to this equation we can apply Lemma 1 with
$$ w = |\widehat{f}(\xi)| \qquad \mbox{and} \qquad z = \frac{\overline{\widehat{f}(\xi)}}{|\widehat{f}(\xi)|} \widehat{g}(\xi)$$
and we obtain
$$ |a|^2 \leq | |\widehat{f}(\xi)| - | \widehat{g}(\xi)||^2 + \frac{1}{5} |b|^2.$$
This, in turn, implies that
$$ |\widehat{f}(\xi) - \widehat{g}(\xi)|^2 = a^2 + b^2 \leq  | |\widehat{f}(\xi)| - | \widehat{g}(\xi)||^2 + \frac{6}{5} |b|^2$$
and thus, recalling the definition of $b$, we get, in this regime, the pointwise estimate
$$ |\widehat{f}(\xi) - \widehat{g}(\xi)|^2 = a^2 + b^2 \leq  | |\widehat{f}(\xi)| - | \widehat{g}(\xi)||^2 + \frac{6}{5} \left| \Imn  \frac{\overline{\widehat{f}(\xi)}}{|\widehat{f}(\xi)|} \widehat{g}(\xi) \right|^2.$$
This concludes our analysis of the first term and we arrive at
\begin{align*}
\int_{|\widehat{f}(\xi)| \geq 10\varepsilon \atop |\widehat{f}(\xi) - \widehat{g}(\xi)| \leq \varepsilon}  | \widehat{f}(\xi) - \widehat{g}(\xi) |^2 d\xi
\leq \| |\widehat{f}| - |\widehat{g}| \|_{L^2(\mathbb{R}^n)}^2 + \frac{6}{5} \left\|  \Imn  \overline{\widehat{f}} |\widehat{f}|^{-1}\widehat{g} \right\|_{L^2(\mathbb{R}^n)}^2.
\end{align*}
\textit{Second term.} We now estimate the second term in the first integral. We recall that if $p=1$, then
$$ \| \widehat{f} - \widehat{g}\|_{L^{\infty}} \leq \|f-g\|_{L^1} = \varepsilon$$
and the domain of integration
$$ X = \left\{\xi \in \mathbb{R}^n: |\widehat{f}(\xi)| \geq 10\varepsilon \wedge |\widehat{f}(\xi) - \widehat{g}(\xi)| \geq \varepsilon \right\}$$
is the empty set. This leads to the slight improvement if $p=1$. Let us thus assume that $1 < p < 2$.
 We use H\"older's inequality to argue that
$$ \int_{X}  | \widehat{f}(\xi) - \widehat{g}(\xi) |^2 d\xi \leq \left( \int_{X} | \widehat{f}(\xi) - \widehat{g}(\xi) |^{\frac{p}{p-1}} d\xi \right)^{\frac{2p-2}{p}} |X|^{\frac{2-p}{p}}.$$
The set $X$ cannot be too big, note that
\begin{align*}
\|\widehat{f}-\widehat{g}\|_{L^{\frac{p}{p-1}}}^{\frac{p}{p-1}} &= \int_{\mathbb{R}^n}{   | \widehat{f}(\xi) - \widehat{g}(\xi) |^{\frac{p}{p-1}} d\xi} \\
&\geq \int_{X}{   | \widehat{f}(\xi) - \widehat{g}(\xi) |^{\frac{p}{p-1}} d\xi} \geq \varepsilon^{\frac{p}{p-1}} |X|.
\end{align*}
The Hausdorff-Young inequality, valid for any $h \in L^p(\mathbb{R}^n)$ for $1 \leq p \leq 2$,
$$ \|\widehat{h}\|_{L^{\frac{p}{p-1}}(\mathbb{R}^n)} \leq \|h\|_{L^p(\mathbb{R}^n)}$$
then implies, recalling the definition of $\varepsilon$,
$$ |X| \leq \frac{\|\widehat{f}-\widehat{g}\|_{L^{\frac{p}{p-1}}}^{\frac{p}{p-1}} }{\varepsilon^{\frac{p}{p-1}}} \leq \frac{\|f-g\|_{L^p}^{\frac{p}{p-1}}}{\varepsilon^{\frac{p}{p-1}}} =1.$$
We remark that we could get a slightly better constant from using Beckner's inequality \cite{babenko, beckner} but do not pursue sharp constants in this paper.
Therefore
$$ \int_{X}  | \widehat{f}(\xi) - \widehat{g}(\xi) |^2 d\xi \leq \left( \int_{X} | \widehat{f}(\xi) - \widehat{g}(\xi) |^{\frac{p}{p-1}} d\xi\right)^{\frac{2p-2}{p}}.$$
Employing the Hausdorff-Young inequality once more, we obtain
\begin{align*}
\int_{X}  | \widehat{f}(\xi) - \widehat{g}(\xi) |^2 d\xi &\leq \left( \int_{X} | \widehat{f}(\xi) - \widehat{g}(\xi) |^{\frac{p}{p-1}} d\xi\right)^{\frac{2p-2}{p}} \\
&\leq \left( \int_{\mathbb{R}^n} | \widehat{f}(\xi) - \widehat{g}(\xi) |^{\frac{p}{p-1}} d\xi\right)^{2\frac{p-1}{p}} 
\leq \| f- g\|_{L^p(\mathbb{R}^n)}^2.
\end{align*}
\textit{Second Integral.} This estimate is simple, we use the elementary inequality
$$ |a-b|^2 \leq ((|a| - |b|) + 2|b|)^2 \leq 2(|a| - |b|)^2 + 8|b|^2.$$
to argue that
\begin{align*}
\int_{|\widehat{f}(\xi)| \leq 10\varepsilon}  | \widehat{f}(\xi) - \widehat{g}(\xi) |^2 d\xi &\leq 8\int_{|\widehat{f}(\xi)| \leq 10\varepsilon}  | \widehat{f}(\xi)  |^2 d\xi \\
&+ 2\int_{|\widehat{f}(\xi)| \leq 10\varepsilon}  | |\widehat{f}(\xi)| - |\widehat{g}(\xi)| |^2 d\xi.
\end{align*}
\textit{Conclusion.} Collecting all these estimates, we obtain two different bounds depending on the value of $p$. If $1 < p < 2$, we obtain
\begin{align*}
 \|f-g\|_{L^2(\mathbb{R}^n)}^2 &\leq 2\cdot \| |\widehat{f}| - |\widehat{g}| \|_{L^2(\mathbb{R}^n)}^2 + \frac{6}{5} \left\|  \Imn  \overline{\widehat{f}} |\widehat{f}|^{-1}\widehat{g} \right\|_{L^2(\mathbb{R}^n)}^2 \\
 &+ \|f-g\|_{L^p(\mathbb{R}^n)}^2  + 8\int_{|\widehat{f}(\xi)| \leq 10\varepsilon}  | \widehat{f}(\xi)  |^2 d\xi.
\end{align*}
If $p=1$, then the $\|f-g\|_{L^1(\mathbb{R}^n)}^2$ term can be omitted because we can estimate the second term in the first integral by 0. 
In either case, recalling the definition of $\varepsilon$,
$$  \int_{|\widehat{f}(\xi)| \leq 10\varepsilon}  | \widehat{f}(\xi)  |^2 d\xi =  \int_{|\widehat{f}(\xi)| \leq 10\|f-g\|_{L^p(\mathbb{R}^n)}}  | \widehat{f}(\xi)  |^2 d\xi$$
which results in the desired statement.
\end{proof}

\subsection{Proof of Corollary 2}

\begin{lemma} Let $f \in L^1(\mathbb{R}^n) \cap L^2(\mathbb{R}^n)$ have its $k-$th derivative in $L^1(\mathbb{R})$ where $k>(n+2)/2$. Then, as $\varepsilon \rightarrow 0$, we have for some constant $c>0$ depending on $f$,
$$  \int_{|\widehat{f}(\xi)| \leq 10\varepsilon}  | \widehat{f}(\xi)  |^2 d\xi \leq c \cdot \varepsilon^{2-n/k}.$$
\end{lemma}
\begin{proof} We know that, for some implicit constant depending only on $f$,
$$ |\widehat{f}(\xi)| \lesssim \frac{1}{1+ |\xi|^k}.$$
We observe that this is $\lesssim \varepsilon$ as soon as $|\xi| \geq \varepsilon^{-1/k}$. This allows us to estimate
\begin{align*}
 \int_{|\widehat{f}(\xi)| \leq 10\varepsilon}  | \widehat{f}(\xi)  |^2 d\xi &\leq  \int_{|\xi| \leq \varepsilon^{-1/k}}  | \widehat{f}(\xi)  |^2  1_{|\widehat{f}(\xi)| \leq 10\varepsilon} d\xi + \int_{|\xi| \geq \varepsilon^{-1/k}}  | \widehat{f}(\xi) |^2 d\xi \\
 &\leq \varepsilon^2 \varepsilon^{-n/k} +  \int_{|\xi| \geq \varepsilon^{-1/k}}  | \widehat{f}(\xi) |^2 d\xi  \\
 &\lesssim \varepsilon^{2-n/k} + \int_{\varepsilon^{-1/k}}^{\infty} |r|^{-2k} r^{n-1} d\xi  \lesssim \varepsilon^{2-n/k}.
 \end{align*}
\end{proof}


\begin{thebibliography}{10}

\bibitem{akut} E. J. Akutowicz. On the determination of the phase of a Fourier integral, I. Transactions of
the American Mathematical Society, pages 179--192, 1956.

\bibitem{akut2} E. J. Akutowicz. On the determination of the phase of a Fourier integral. II. Proc. Amer.
Math. Soc., 8:234--238, 1957.

\bibitem{al} R. Alaifari and P. Grohs, Gabor phase retrieval is severely ill-posed, App. Comp. Harm. Anal, to appear

\bibitem{al2} R. Alaifari, I. Daubechies, P. Grohs and G. Thakur, Recunstructing real-valued functions from
unsigned coefficients with respect to wavelet and other frames, Journal of Fourier Analysis and Applications 23 (2017), p. 1480--1494 

\bibitem{alex} B. Alexeev, A. S. Bandeira, M. Fickus and D. G. Mixon, Phase retrieval with polarization
SIAM Journal on Imaging Sciences 7 (2014), p. 35--66

\bibitem{babenko} I. Babenko,  An inequality in the theory of Fourier integrals, Izvestiya Akademii Nauk SSSR. Seriya Matematicheskaya, 25 (1961): 531--542

\bibitem{bar}  A. Barnett, C. Epstein, L. Greengard and J. Magland, Geometry of the Phase Retrieval Problem, arXiv:1808.10747 

\bibitem{beckner} W. Beckner, William, Inequalities in Fourier analysis, Annals of Mathematics, Second Series, 102 (1975): p. 159--182

\bibitem{bal} R. Balan, B. G. Bodmann, P. G. Casazza, and D. Edidin. Painless reconstruction from
magnitudes of frame coefficients. J. Fourier Anal. Appl., 15 (2009): pp. 488--501.

\bibitem{bal2} R. Balan, P. Casazza, and D. Edidin. On signal reconstruction without phase. Appl. Comput.
Harmon. Anal., 20(3):345--356, 2006.

\bibitem{bal3} R. Balan and Y. Wang. Invertibility and robustness of phaseless reconstruction. Appl. Comput. Harmon. Anal., 38(3):469--488, 2015.

\bibitem{ban} A. S. Bandeira, J. Cahill, D. G. Mixon, and A. A. Nelson. Saving phase: Injectivity and
stability for phase retrieval. Appl. Comput. Harmon. Anal., 37(1):106--125, 2014.

\bibitem{bu} R. E. Burge, M. A. Fiddy, A. H. Greenaway, and G. Ross. The application of dispersion
relations (Hilbert transforms) to phase retrieval. Journal of Physics D: Applied Physics,
7(6):65, 1974.

\bibitem{bu2} R. E. Burge, M. A. Fiddy, A. H. Greenaway, G. Ross, and W. C. Price. The phase problem. Proceedings of the Royal Society of London. A. Mathematical and Physical Sciences,
350(1661):191--212, 1976.

\bibitem{bruck} Y. Bruck and L. Sodin, On the Ambiguity of the Image Reconstruction Problem, Optics Communication 30 (1979), p. 304--308.

\bibitem{cahill} J. Cahill, P. G. Casazza and I. Daubechies, Phase retrieval in infinite-dimensional Hilbert spaces. Trans. Amer. Math. Soc. Ser. B 3 (2016), p. 63--76.

\bibitem{cr} T. R. Crimmins and J. R. Fienup. Ambiguity of phase retrieval for functions with disconnected support. J. Opt. Soc. Am., 71 (1981): p. 1026--1028.

\bibitem{cr2} ] T. R. Crimmins and J. R. Fienup. Uniqueness of phase retrieval for functions with sufficiently disconnected support. J. Opt. Soc. Am., 73(1983): p. 218--221.

\bibitem{dainy} J. Dainty and J. Fienup. Phase retrieval and image reconstruction for astronomy. Image
Recovery: Theory Appl., 13:231--275, 1987.

\bibitem{eldar} Y. Eldar, P. Sidorenko, D. G. Mixon, S. Barel and O. Cohen, Sparse Phase Retrieval from Short-Time Fourier Measurements, IEEE Signal Processing Letters 22.5 (2014): 638--642.

\bibitem{fann} A. Fannjiang, Absolute uniqueness of phase retrieval with random illumination, Inverse Problems 28 (2012), 075008

\bibitem{fann2} A. Fannjiang and T. Strohmer, The Numerics of Phase Retrieval, Acta Numerica, to appear

\bibitem{fienup} J. R. Fienup. Reconstruction of an object from the modulus of its Fourier transform. Opt.
Lett., 3(1978): p. 27--29.

\bibitem{gabor} D. Gabor. A New Microscopic Principle. Nature, 161 (1948): p. 777--778.

\bibitem{green} A. H. Greenaway. Proposal for phase recovery from a single intensity distribution. Opt.
Lett., 1 (1977):  p.10--12.

\bibitem{grohs} P. Grohs, S. Koppensteiner and M. Rathmair, Phase Retrieval: Uniqueness and Stability,
SIAM Review, to appear.

\bibitem{grohs1} P. Grohs and M. Rathmair, Stable Gabor Phase Retrieval and Spectral Clustering, Comm. Pure Appl. Math 72 (2019):p. 981--1043.

\bibitem{grohs2} P. Grohs and M. Rathmair,  Stable Gabor phase retrieval for multivariate functions, arXiv:1903.01104

\bibitem{ha} R. Harrison, Phase problem in crystallography, J. Opt. Soc. Am. A
10 (1993), p. 1045--1055.

\bibitem{hof} E. Hofstetter. Construction of time-limited functions with specified autocorrelation functions. IEEE Transactions on Information Theory, 10 (1964): p. 119--126.

\bibitem{jam} P. Jaming. Phase retrieval techniques for radar ambiguity problems. Journal of Fourier
Analysis and Applications 5 (1999): p. 309--329.


\bibitem{jam1} P. Jaming Uniqueness results in an extension of Pauli's phase retrieval.
Applied and Computational Harmonic Analysis 37 (2014) 413--441.


\bibitem{jam2} P. Jaming, K. Kellay and R. Perez, Phase Retrieval for Wide Band Signals, 13th International conference on Sampling Theory and Applications (SampTA), Bordeaux, France, 2019,pp 1--4.




\bibitem{kil} M. V. Klibanov, P. E. Sacks, and A. V. Tikhonravov. The phase retrieval problem. Inverse Problems, 11(1):1, 1995.

\bibitem{mallat} S. Mallat and I. Waldspurger. Phase retrieval for the Cauchy wavelet transform. J. Fourier
Anal. Appl., 21: p. 1251--1309, 2015.

\bibitem{miao} J. Miao, D. Sayre, and H. N. Chapman, Phase retrieval from the magnitude of the Fourier transforms of nonperiodic objects, J. Opt. Soc. Am. A 15 (1998), p. 1662--1669 

\bibitem{mill} R. Millane), Phase retrieval in crystallography and optics, J. Opt. Soc. Am.
A. 7 (1990), 394--411.
 
\bibitem{rose} J. Rosenblatt. Phase retrieval. Comm. Math. Phys., 95: p.317--343, 1984.

\bibitem{sanz} J. Sanz, Mathematical Considerations for the Problem of Fourier Transform Phase Retrieval from Magnitude, SIAM J. Appl. Math., 45 (1985), 651--664. 

\bibitem{sh} Y. Shechtman, Y. C. Eldar, O. Cohen, H. N. Chapman, J. Miao and M. Segev, Phase retrieval with application to optical imaging: a contemporary
overview, IEEE Signal Processing Magazine 32 (2015),  p. 87--109.

\bibitem{walther}  A. Walther. The question of phase retrieval in optics. Journal of Modern Optics, 10(1):41--49, 1963.
\end{thebibliography}
\end{document}